\documentclass{article}

\usepackage[hypertexnames=false]{hyperref}
\usepackage{color}
\usepackage{latexsym}
\usepackage{dsfont}
\usepackage{amssymb}
\usepackage{float}
\usepackage{algorithm,algorithmic}
\usepackage{graphicx}
\usepackage{amsmath, amsfonts,amssymb,theorem,euscript,array,enumerate,amsfonts,mathrsfs}
\usepackage{subcaption}
\usepackage{ulem}
\usepackage{bm, bbm}

\usepackage{endnotes}
\let\footnote=\endnote
\let\enotesize=\normalsize
\def\notesname{Endnotes}%
\def\makeenmark{$^{\theenmark}$}
\def\enoteformat{\rightskip0pt\leftskip0pt\parindent=1.75em
\leavevmode\llap{\theenmark.\enskip}}

\usepackage{natbib}
\bibpunct[, ]{(}{)}{,}{a}{}{,}%
\def\bibfont{\small}%

\newtheorem{Theorem}{Theorem}[part]
\newtheorem{Definition}{Definition}[part]
\newtheorem{Proposition}{Proposition}[part]

\newtheorem{Lemma}{Lemma}[part]

\newtheorem{Remark}{Remark}[part]
\newtheorem{Example}{Example}[part]

\usepackage[T1]{fontenc}
\usepackage[latin1]{inputenc}
\usepackage{hhline}
\usepackage{graphicx}
\usepackage{verbatim}
\usepackage{eurosym}
\usepackage{dsfont}

\theorembodyfont{\upshape}
\newcommand{\nc}{\newcommand}
\nc{\ind}{\mathds{1}}

\newcommand{\R}{\mathbb{R}}

\newcommand{\E}{\mathcal{E}}

\DeclareMathOperator{\esssup}{esssup}

\def\esssup_#1{\underset{#1}{\mathrm{ess\,sup\, }}}
\def\essinf_#1{\underset{#1}{\mathrm{ess\,inf\, }}}
\def\argmax_#1{\underset{#1}{\mathrm{arg\,max\, }}}
\def\argmin_#1{\underset{#1}{\mathrm{arg\,min\, }}}

\def\reff#1{{\rm(\ref{#1})}}
\def \ep{\hbox{ }\hfill$\Box$}

\def \Min{\displaystyle\min}

\def\b1{\bf 1}

\def \A{\mathbb{A}}

\def \N{\mathbb{N}}
\def \R{\mathbb{R}}

\def \E{\mathbb{E}}

\def \P{\mathbb{P}}

\def \A{{\cal A}}
\def \Bc{{\cal B}}

\def \Fc{{\cal F}}

\def \Hc{{\cal H}}

\def \Pc{{\cal P}}
\def \Mc{{\cal M}}

\def \Sc{{\cal S}}

\def \Wc{{\cal W}}

\def \ep{\hbox{ }\hfill$\Box$}

\renewcommand{\theequation}{\thesection.\arabic{equation}}

\numberwithin{equation}{section}

\newenvironment{proof}[1][{\it Proof.}]{\begin{trivlist}
\item[\hskip \labelsep {\bfseries #1}]}{ \hfill
$\Box$\end{trivlist}\vskip -0.2 cm}

\def\reff#1{{\rm(\ref{#1})}}
\def\beqs{\begin{eqnarray*}}
\def\enqs{\end{eqnarray*}}
\def\beq{\begin{eqnarray}}
\def\enq{\end{eqnarray}}

\begin{document}
\title{Dynamic Programming Principles for Mean-Field Controls with Learning}

\author{Haotian Gu
\thanks{Department of Mathematics, University of California, Berkeley, USA. \textbf{Email:} haotian$\_$gu@berkeley.edu }
\and
Xin Guo
\thanks{Department of Industrial Engineering \& Operations Research, University of California, Berkeley, USA. \textbf{Email:} xinguo@berkeley.edu}
\and
Xiaoli Wei
\thanks{Tsinghua-Berkeley Shenzhen Institute, Shenzhen, China. \textbf{Email:} xiaoli\_wei@sz.tsinghua.edu}
\and
Renyuan Xu
\thanks{Industrial  \& Systems Engineering, University of Southern California, Los Angeles, USA. \textbf{Email:} renyuanx@usc.edu}
}



\maketitle

\begin{abstract}
Dynamic programming principle (DPP) is fundamental for control and optimization, including Markov decision problems (MDPs), reinforcement learning (RL), and more recently mean-field controls (MFCs). 
However, in the learning framework of MFCs, DPP has not been rigorously established, despite its critical importance for algorithm designs.
In this paper, we first present a simple example in MFCs with learning where DPP fails with a mis-specified Q function; and then propose the correct form of Q function in an appropriate space
for MFCs with learning. This particular form of Q function is different from the classical one and is called the IQ function. In the special case when the transition probability and the reward are independent of the mean-field information,  it {\it integrates} the classical Q function for single-agent RL over the state-action distribution. In other words, MFCs with learning can be viewed as lifting the classical RLs by replacing the
state-action space with its probability distribution space. This identification of  the IQ function enables us to establish precisely the DPP  in  the learning framework of MFCs. Finally, we illustrate through numerical experiments the time consistency of this IQ function.
\end{abstract}

\section{Introduction}

\paragraph{DPP.}

Widely regarded as one of the fundamental principles for control and optimization, dynamic programming principle (DPP) was first established for value functions of Markov decision problems (MDPs) in \cite{bellman1957markovian}, and later for more general frameworks in \cite{bertsekas1978stochastic} and \cite{fleming2006controlled}. DPP was also established for the Q functions in a learning framework of MDP in \cite{WATKINS1989}
(see \cite{bertsekas1996neuro} and \cite{sutton2018reinforcement}).
 The DPP implies the time consistency property of the optimal control in that a current optimal policy remains so for the future. This time consistency is critical for reinforcement learning  (RL): for model-free learning,  time consistency of the Q function is the key apparatus for Q-learning algorithms (\cite{WD1992} and \cite{mnih2015human}) and for the actor-critic approach (\cite{KT2000} and \cite{lillicrap2015continuous}); for model-based learning, time consistency of the value function lays the foundation for value iteration and policy based algorithms (\cite{Doya2000,doya2002multiple}). 
  More recently, the time consistency property has been analyzed in a series of papers for mean-field controls (MFCs) also known as McKean-Vlasov (MKV) controls
  (\cite{LP2014}, \cite{PW2016}, and  \cite{DPT2019}), {without the context of learning.}
\paragraph{MKV controls/MFCs.} 
McKean-Vlasov (MKV) processes, first introduced and  studied in \cite{MCKEAN1969},  are stochastic processes governed by  stochastic differential equations whose coefficients  depend on distributions of the solutions. MKV controls concern optimal controls of  such systems where interchangeable agents interact through the distribution of their states {and actions}. As such, MKV controls are often called mean-field controls (MFCs). 

From the game theory perspective, 
MFCs are closely related  to  mean-field games (MFGs).
Both are stochastic games with infinite number of agents, with MFGs  the limiting regime of games under Nash equilibrium and  MFCs that of games by Pareto optimality. 
Theories of MFGs and MFCs have progressed rapidly and have been adopted  in a number of fields such as physics, economics, and data science. (See \cite{LL2007},
\cite{BFY2013, CD2018}, and \cite{HMC2006}). 
MFCs, in particular,  have been  broadly applied to model collective behaviors of stochastic systems with
 a large number of mutually interacting agents, including
 \cite{GPY2013} for systemic risk assessment,    \cite{NUNO2017} for a large benevolent planner such as the government or the central bank to control taxes or interest rates, and
\cite{ABP2017} for consumers to choose between new energy resources and traditional ones.

\paragraph{MFCs with learning and DPP.}
For many of the stochastic systems {with a large population of agents}, model parameters and dynamics  are in general unknown {\it a priori}  {\color{black}and learning algorithms are essential for the agents to improve their decisions {while} interacting with the (partially) unknown system and other agents. In this case, multi-agent reinforcement learning (MARL) has enjoyed substantial successes for analyzing the otherwise challenging games, including two-agent or two-team computer games \citep{SHMGSVSAPL2016,VBCMJCDHGP2019}, self-driving vehicles \citep{SSS2016}, real-time bidding games \citep{JSLGWZ2018}, ride-sharing \citep{LQJYWWWY2019}, and traffic routing \citep{EAA2013}.
 Despite its empirical success, MARL suffers from the curse of dimensionality known also as the {\it combinatorial nature} of MARL: its sample complexity by  existing algorithms for stochastic dynamics grows exponentially with respect to the number of agents $N$.
 In practice, this $N$ can be on the scale of thousands or more, for instance, in rider match-up for Uber-pool and network routing for Zoom.
MFCs, on the other hand, provide good approximations to the  multi-agent system and address the curse of dimensionality suffered in most of the existing MARL algorithms. It is therefore natural meanwhile important} to consider the learning problem in  MFCs. 

{Despite its potential for improving existing MARL algorithms}, theory of MFCs with learning remains by and large undeveloped. Instead, almost all  works  (for example, \cite{CLT2019a}, \cite{CLT2019b} and \cite{WYW2019}) focus mainly on learning algorithms while assuming heuristically some forms of DPP. 


It is tempting to assume  DPP  given the similarity between MFCs and MDPs. 
Yet, MFCs are  fundamentally  different from MDPs:   MKV systems depend on  marginal distributions of both the state and the control. Consequently, MFCs  are inherently time inconsistent. For instance,
 it has been well recognized that DPP in general does not hold for the controlled MKV system  due to its non-Markovian nature (\cite{AD2011}, \cite{BDL2011}, and \cite{CD2015}).
Only recently, this time inconsistency issue for MFCs was resolved by appropriately enlarging state spaces, for example, in  \cite{LP2014} and \cite{PW2016}  for a finite time horizon and in \cite{DPT2019} for a more general framework. 
When MFC is coupled with learning,  it is unclear if, when, and how DPP will hold. This is the focus of this paper. 

\paragraph{Time consistency in MFCs with learning.}
In this paper, we will first present a simple example (Example \ref{wrong_Q} in Section \ref{wrong_Q})  to demonstrate the time inconsistency issue for  MFCs with learning.  This example shows that when the  Q function is mis-specified, Q table will converge to different values with different initial population distributions.  

We will then establish precisely the DPP by identifying a correct form of the Q function in an appropriate space. This particular form of the Q function reflects the essence of MFCs: MFC is equivalent to  an auxiliary control problem in which the  objective function depends on the cost functional of every agent for the purpose of social optimality. This control perspective enables us to specify the Q function as an integral form of the classical Q function over the state-action distribution of each agent. To distinguish such Q function from the classical one, we  called it integrated Q (IQ) function. (See also Section \ref{I-remark}). 

Next, we derive the suitable form of DPP for this {IQ} function.
This DPP  generalizes  the classical DPP for Q-learning of MDP  to that of  MKV system, and extends  the DPP for MFCs to the learning framework. 
To accommodate model-based learning for MFCs, we also obtain the corresponding DPP for the value function.

Finally, we illustrate through numerical experiments the time consistency of the {IQ} function.

\paragraph{Relation to existing works.}
Our analysis and framework for establishing DPP for MFCs with learning differ fundamentally  from those in \cite{LP2014}, \cite{PW2016}, and  \cite{DPT2019} on DPP for value functions of MFCs without learning.

 The first is the adoption of  relaxed controls  instead of   strict controls  used in these earlier works.  As illustrated in Example \ref{wrong_Q} in Section \ref{wrong_Q}, optimal controls for MFCs with learning are intrinsically relaxed types, whereas classic control problems with concave reward functions are inevitably strict even for MFGs (see \cite{Lacker2015}).
 Relaxed controls are essential for learning, and in particular for RL which is characterized with exploration and exploitation. Exploration relies on   randomized strategies  with actions sampled from a distribution of actions. Relaxed controls are  known also as  mixed strategies in game theory  (\cite{WLB2010,DFR1998}, and  \cite{MT2013}), also for MFC without learning in \cite{LACKER2017}.
 Moreover, incorporation of  entropy regularization in many machine learning problems would destroy the convexity or the concavity structure  of the value function,  and  optimal controls  are necessarily  relaxed ones. 
 
  The {second} is  the aforementioned {IQ} function, identified and analyzed for the first time in the learning framework  on the lifted probability measure space with relaxed controls.
 
 To the best of our knowledge, this is the first time that DPP is rigorously established for MFCs with learning.
 This form of DPP provides one critical insight: learning problems with  MFCs  can be recast as  general forms of RLs where the state variable is replaced by  the probability distribution. This reformulation  paves the way for developing efficient value-based and policy-based algorithms for MFCs with learning. It is also the first step towards future theoretical development of learning problem with MFCs. 
 For instance,  \cite{MP2019} has further established the DPP for learning in a discrete-time model with the incorporation of  common noise and with open-loop controls.

 \paragraph{Outline of the paper.} The rest of the paper is organized as follows.
 Section $2$ presents the mathematical framework of MFCs with learning. Section $3$ introduces the {IQ} function and establishes DPPs for both the {IQ} function and the value function.
Section $4$ concludes by revisiting   Example \ref{wrong_Q} with the performance of the {IQ} function. {{Section $5$ demonstrates an example on equilibrium pricing with {IQ} function.}}

\paragraph{Notations.} 

\begin{itemize}
    \item  For a measurable space $(X, \Fc_X)$, $\Pc(X)$ denotes the space of all probability measures on $(X, \Fc_X)$. When $X$ is a {\it compact} topological space, $\Bc(X)$ denotes its Borel $\sigma$-field, $\Pc(X)$ is endowed with the topology of weak convergence,
{and hence any probability measure $\mu$ $\in$ $\Pc(X)$ has a first moment}. $\Wc_1$ denotes the Wasserstein distance of order $1$ such that
\beqs
\Wc_1(\mu, \mu')= \inf\left\{\biggl(\int_{X \times X}d_{X}(x, x')\nu_{}(dx, dx')\biggl):  \nu_{} \in \Pc(X \times X) \mbox{ with marginals } \mu, \mu' \in \Pc(X)\right\}.
\enqs
$\Pc(X)$ is always equipped with $\Wc_1(\mu, \mu')$. 
The Borel $\sigma$-field of $\Pc(X)$ is $\sigma$-field induced by the evaluation $\Pc(X) \ni \mu \mapsto \mu(C)$ for any {Borel set} $C \subset X$. Note that the Borel $\sigma$-field of $\Pc(X)$ is generated by $\Wc_1$. (See e.g. \cite{Villani2009} and \cite{Lacker2015}).
\item When $X$ is finite, $\Pc(X)$ $=$ $\left\{(p_i)_{i=1}^{|X|} \in \R^{|X|}: \sum_{i=1}^{|X|} p_i =1, p_i \geq 0\right\}$ is the probability simplex in $\R^{|X|}$, where $|X|$ denotes the size of $X$; Moreover, $X$ is always equipped with discrete metric, i.e., $d(x, x') = {{\bf 1}}_{\{x \neq x'\}}$. {In this case,  $\Wc_1$ is equivalent to the $L^1$-norm. (See e.g. \cite{GS2002}).}
\item For a metric space $X$, $\Mc(X)$ denotes the set of all real-valued measurable functions on $X$. For each bounded $f \in \Mc(X)$, the sup norm of $f$ is defined as $\|f\|_\infty= \sup_{x \in X} |f(x)|$. 
\item Denote $(\Omega, \Fc= \{\Fc_t\}_{t=0}^\infty, \P)$ as a probability space with $\Omega$ being Polish space, $\Fc$ its Borel $\sigma$ field and $\P$ an atomless probability measure, and denote $L(\Omega, \Fc, \P; X)$ as the space of all X-valued random variables; $(\Omega, \Fc= \{\Fc_t\}_{t=0}^\infty, \P)$ is ``rich'' in the sense that for any $\mu \in \Pc(X)$, there exists $\xi \in L(\Omega, \Fc, \P; X)$ satisfying $\xi \sim \mu$. Given another measurable space $(Y, \Fc_Y)$, we say a measure-valued function $f: Y \to  \Pc(X)$ is measurable if if the evaluation $f(C): Y \to \R$ is measurable for any $C \in \Fc_X$.
\item {Given two measurable spaces $(X, \Fc_X)$ and $(Y, \Fc_Y)$, for a measurable function $f: X \to Y$ and a measure $\mu \in \Pc(X)$, the pushforward measure $f \star \mu$ is defined to be a probability measure $\Fc_Y \to \R$: $f \star \mu(C) = \mu(f^{-1}(C))$ for any $C \in \Fc_Y$.}
\item {{Given a metric space $X$, $\delta_x$ denotes the Dirac measure on some fixed point $x \in X$. $\N$ denotes the set of all positive integers.}}
\end{itemize}

\section{The Mathematical Framework of MFCs with Learning}
 \subsection{Review: Single-Agent Reinforcement Learning} \label{secsclassicalRL}
 Before introducing the mathematical framework of MFCs with learning, 
 let us review relevant terminologies for single-agent RL.
 
 Let us  start with a discrete time MDP in an infinite time horizon of the following form
 \beq \label{equ: singlev}
v(s) = \sup_{\pi} v^{\pi}(s):=
 \sup_{\pi}\E^\pi\biggl[\sum_{t=0}^\infty \gamma^t r(s_t, a_t) \biggl| s_0 = s\biggl],
 \enq
 \vspace{-0.5cm}
 subject to
 \vspace{-0.5cm}
 \beq \label{equ: singledynamics}
 s_{t + 1} \sim P(s_t, a_t), \;\;\; a_t \sim \pi_t(s_t),\;\; t \in \N \cup \{0\}.
 \enq
Here and throughout the paper, $\E^{\pi}$ denotes the expectation under control ${\pi}$; the state space $(\Sc, d_{\Sc})$ and the action space $(\A, d_{\A})$ are two compact separable metric space , including the case of  $\Sc$ and $\A$  being finite, as often seen in RL literature; $\gamma$ $\in$ $(0, 1)$ is a  discount factor; $s_t \in \Sc$ and $a_t \in \A$ are the state and the action at time $t$; $P: \Sc \times \A \to \Pc(\Sc)$ is the transition matrix of the underlying Markov system; the reward  $r(s, a)$ is random valued in $\R$ for each $(s, a) \in \Sc \times \A$; and the control $\pi =\{\pi_t\}_{t=0}^\infty$ can be either deterministic such that  $\pi_t: \Sc \to \A$,  or randomized such that $\pi_t: \Sc \to \Pc(\A)$. {Note that our results can be easily extended to the situation where $(\mathcal{S},d_{\mathcal{S}})$ and $(\mathcal{A},d_{\mathcal{A}})$ are  not  compact  but  the  measures  under  consideration  have  a  first  moment.}

When the transition dynamics $P$ and the reward function $r$ are unknown, this MDP becomes an RL problem, which is to find an optimal control $\pi$ (if it exists) while simultaneously learning the unknown $P$ and $r$. The learning of $P$ and $r$ can be either explicit or implicit, which leads to model-based and model-free RL, respectively.

One basic model-free algorithm for  RL  is the Q-learning algorithm,  in  which a Q function is defined as
\beq \label{defsingleQ}
Q(s, a) = \E[r(s, a)] + \gamma \E_{s' \sim P(s, a)}[v(s')].
\enq
The well-known DPP for such Q function is expressed in the form of the following Bellman equation
\beq \label{equ: singleQ}
Q(s, a) = \E[r(s, a)] + \gamma \E_{s' \sim P(s, a)}\sup_{a' \in \A} Q(s', a').
\enq
Meanwhile, the Bellman equation for the value function is 
\beq \label{equ:classicalV}
v(s) = \sup_{a \in \A} \big\lbrace \E[r(s, a)] + \gamma \E_{s' \sim P(s, a)}[v(s')] \big\rbrace.
\enq
By the definition of Q function and \reff{equ:classicalV}, the value function  and the 
Q function are closely connected by the following relation
$$v(s) = \sup_{a \in \A} Q(s, a).$$ 
Thus, one can retrieve the 
optimal (stationary) control $\pi^*(s)$ (if it exists) from  $Q(s, a)$, i.e.,  $\pi^*(s) \in \arg\max_{a \in \A} Q(s, a)$.

\subsection{Mathematical Framework  of MFCs with Learning} \label{se:MDPMKV}
Our MFC framework is motivated from cooperative $N$-agent games. To see this, assume that there are $N$ homogeneous agents. At each time step $t \in \N \cup \{0\}$, the state and the action of each agent $i$ $(=1, \cdots, N)$ is denoted as $s_t^i \in \Sc$ and $a_t^i \in \A$. Each agent $i$ moves to the next state $s_{t + 1}^i$ according to the transition probability $P(s_t^i, \mu_t^N, a_t^i, \nu_t^N)(\cdot)$ and receives a reward $r_t^i$ $\sim$ $R(s_t^i, \mu_t^N, a_t^i, \nu_t^N)(\cdot)$, where $\mu_t^N = \frac{1}{N} \sum_{i=1}^N\delta_{s_t^i}$ and  $\nu_t^N = \frac{1}{N} \sum_{i=1}^N \delta_{a_t^i}$ are the empirical distributions of $s_t^i$ and $a_t^i$, $i=1, \ldots, N$;  the probability transition $P$ is a measurable function from $\Sc \times \Pc(\Sc) \times \A \times \Pc(\A)$ to $\Pc(\Sc)$ and is unknown; and the distribution of the reward function $R:$ $\Sc \times \Pc(\Sc) \times \A \times \Pc(\A)$ $\to$ $\Pc(\R)$ is measurable and unknown.

Now, taking $N \to \infty$, by law of large number, we can consider MFCs, which are stochastic games under Pareto optimality with infinitely many
 identical, indistinguishable, and interchangeable agents. We can define analogously the learning framework for MFCs over the infinite horizon with the same terminology  $\Sc$, $\Pc(\Sc)$, $\A$, $\Pc(\A)$, {$R$}, and $\gamma$ used in the RL framework.  Due to the indistinguishability of agents, one can focus on a single representative agent 
and consider an auxiliary control problem in which the  objective function depends on the average cost/reward  of every agent.

At each time $t \in \N \cup \{0\}$, the state of the representative agent is $s_t$ $\in$ $\Sc$. Given the population state distribution, i.e., the probability distribution $\mu_t$ $\in$ $\Pc(\Sc)$ of state $s_t$, the representative agent takes an action $a_t$ $\in$ $\A$ according to some control $\pi_t$. She will receive an instantaneous stochastic reward $r_t=r(s_t, \mu_t, a_t, \nu_t)$ $\sim R(s_t, \mu_t, a_t, \nu_t)(\cdot)$ and her state will move to the next state $s_{t + 1}$ according to a probability transition function of mean-field type $P(s_t, \mu_t, a_t, \nu_t)(\cdot)$. {Here $\nu_t \in \Pc(\A)$} denotes the action distribution at time $t$. 

The (accumulated) reward of the auxiliary control problem, given the initial state 
$s_0 = \xi$ $\in L(\Omega, \Fc, \P; \Sc)$,  and given the control $\pi=\{\pi_t\}_{t=0}^\infty$, is defined as
\beq \label{eq:vpidef}
{V}^{\pi}(\xi) &=& \E^{\pi}\biggl[\sum_{t=0}^{\infty} \gamma^t r(s_t, \mu_t, a_t, \nu_t)  \biggl| s_0 = \xi\biggl],
\enq
\vspace{-0.3cm}
subject to
\vspace{-0.3cm}
\beq \label{eq: dynamicdef}
s_{t + 1} \sim P(s_t, \mu_t, a_t, \nu_t)(\cdot), \;\;\;a_t \sim \pi_t(s_t,\mu_t)(\cdot), \;\;\; {{r(s_t, \mu_t, a_t, \nu_t) \sim R(s_t, \mu_t, a_t, \nu_t)(\cdot)}}.
\enq


The admissible controls are of feedback forms and relaxed types. That is, at  each time $t$, $\pi_t= \pi_t(s_t, \mu_t)$ and  $\pi_t:$ $\Sc$ $\times$ $\Pc(\Sc)$ $\to$ $\Pc(\A)$ is measurable and maps the current state and the current state distribution to a distribution over the action space. We denote by $\Pi_t$ such set of admissible controls starting from time  $t\in \N \cup \{0\}$, and set $\Pi= \Pi_0$. Note that a relaxed control differs from a strict control, which is {{a measurable function}} defined from $\Sc$ $\times$ $\Pc(\Sc)$ to $\A$. Clearly a strict control $\alpha_t$ is a relaxed control with a special form of $\pi_t$ $=$ $\delta_{\alpha_t}$, the point mass at some {{measurable function}} $\alpha_t$ $:$ $\Sc$ $\times$ $\Pc(\Sc)$ $\to$ $\A$. {Note that under a feedback relaxed control $\pi_t$, we have $\nu_t(\cdot)=\int_{s \in \mathcal{S}} \pi_t(s, \mu_t)(\cdot)\mu_t(ds)\in\Pc(\A)$.}

The objective of the auxiliary controller is to  find  
\beq \label{equ: defv}
V(\xi) &=& \sup_{{\pi} \in \Pi} V^{\pi}(\xi), \;\;\; \text{for any}\; \xi \in L(\Omega, \Fc, \P; \Sc),
\enq
and to search for an optimal control (if it exists).

Note that the nature of MFCs is different from the single-agent RL \reff{equ: singlev}-\reff{equ: singledynamics} in that it reflects the nature of MFC that the representative agent interacts with all agents via probability distributions of states $\mu_t$ and actions $\nu_t$.

To ensure  the well-definedness of this learning problem for MFC \reff{eq:vpidef}-\reff{equ: defv}, throughout the paper we assume:

{{\bf Outstanding Assumption} ({\bf{A}}).} \label{ass: wellposedness} {{For any initial state $s_0 = \xi$ $\sim \mu$,}}
\beqs
\sup_{\pi \in \Pi}\E^{\pi}\Big[\sum_{t=0}^\infty \gamma^t \big|r(s_t, \mu_t, a_t, \nu_t)\big|\Big] & < & \infty.
\enqs

It is clear that when $\|r\|_\infty \leq r_{\text{max}}$, a.s. for some $r_{\text{max}} > 0$, condition in Outstanding assumption ({{A}}) is satisfied. In general, the following conditions (A1)-(A3) will  ensure  Outstanding Assumption ({{A}}). 

\begin{description}
\item[(A1)] For fixed arbitrary $(s^o, \delta_{s^o},{a^o, \delta_{a^o}})$ $\in$ $\Sc$ $\times$ $\Pc(\Sc)$ $\times \A \times \Pc(\A)$,
there exists some positive constant $L_P$ such that for every $(s, \mu, a, {{\nu}})$ $\in$ $\Sc$ $\times$ $\Pc(\Sc)$ $\times$ $\A$ $\times \Pc(\A)$,
\beqs
& & \int_{s' \in \Sc} d_{\Sc}(s', s^o)\Big(P(s, \mu, a, \nu)(ds'){- P(s^o, \delta_{s^o}, a^o, \delta_{a^o})(ds')}\Big)\\
&\leq & L_P \Big( d_{\Sc}(s, s^o) + {d_\A(a, a^o)}+ \Wc_1(\mu, \delta_{s^o}) + \Wc_1(\nu, \delta_{a^o}) \Big).
\enqs
\item[(A2)] For fixed arbitrary $(s^o, \delta_{s^o}, a^o, \delta_{a^o})$ $\in$ $\Sc$ $\times$ $\Pc(\Sc)$ $\times$ $\A$ $\times$ $\Pc(\A)$,
there exists some positive constant $L_R$ such that for every $(s, \mu, a, \nu)$ $\in$ $\Sc$ $\times$ $\Pc(\Sc)$ $\times$ $\A$ $\times \Pc(\A)$,
\beqs
\int_{\R} |r| \Big(R(s, \mu, a, \nu)(dr)- R(s^o, \delta_{s^o}, a^o, \delta_{a^o})(dr)\Big) \leq L_R \Big( d_{\Sc}(s, s^o) + d_\A(a, a^o) + \Wc_1(\mu, \delta_{s^o}) + \Wc_1(\nu, \delta_{a^o})\Big).
\enqs
\item [(A3)] For fixed arbitrary $(s^o, \delta_{s^o}, a^o)$ $\in$ $\Sc$ $\times$ $\Pc(\Sc)$ $\times$ $\A$,
there exists some positive constant $L_\pi$ such that for every $(s, \mu)$ $\in$ $\Sc$ $\times$ $\Pc(\Sc)$ 
\beqs
&& \int_{a \in \A} d_{\A}(a, a^o) \Big(\pi(s, \mu)(da) - \pi(s^o, \delta_{s^o})(da)\Big) \leq L_{\pi} \Big( d_{\Sc}(s, s^o) + \Wc_1(\mu, \delta_{s^o})\Big), \\
&& \int_{a \in \A}d_{\A}(a, a^o) \pi(s^o, \delta_{s^o})(da) < + \infty. 
\enqs
\end{description}

In the MFC formulation, it is important to view alternatively the control  $\pi_t $ as a {{measurable}} mapping from $\Pc(\Sc)$ to $\Hc$. For notational simplicity, set $h_t: =$ $\pi_t(\mu_t) \in \Hc$,  where
\begin{eqnarray}\label{H_space}
\Hc=\{h \;|\; h: \Sc \to \Pc(\A) \; \text{is measurable} \;\; \text{such that Outstanding Assumption (A) holds} \}.
\end{eqnarray}
Here $\Hc$ contains all ``local'' policies that depend only on the state variable. 

\medskip

We first show that the probability distribution of the dynamics $\{\mu_t\}_{t=0}^\infty$  in \reff{eq: dynamicdef} satisfies the following flow property. 
\begin{Lemma}  \label{lem: flowmut} (Flow property of $\mu_t$) Under Outstanding Assumption (A), for any given admissible policy $\pi \in \Pi$ and the initial state distribution $\mu$, the evolution of the state distribution $\{\mu_t\}_{t=0}^\infty$  in \reff{eq: dynamicdef} follows
\beq \label{eqv:flowmut}
\mu_{t + 1}  &=& \Phi(\mu_t, \pi_t(\mu_t)),  \; {{\mu_0 = \mu}},\;\; t \in \N \cup \{0\}.
\enq
Here $\Phi: \Pc(\Sc) \times \Hc \to \Pc(\Sc)$ {is measurable} and defined by
\beq \label{equ: Phi}
\Phi(\mu, h)(ds') &: =& \int_{s \in \Sc} \mu(ds) \int_{a \in \A} h(s)(da) P(s, \mu, a, \nu(\mu, h))(ds')
\enq
for any $(\mu, h) \in \Pc(\Sc) \times \Hc$ and {\color{black}{${\nu}(\mu, h)(\cdot) := \int_{s \in \mathcal{S}} h(s)(\cdot)\mu(s) \in \Pc(\A)$}}. 
In particular, when $h_t = \delta_{\alpha_t}$ for some {{measurable function}} $\alpha_t:$ $\Sc$ $\to$ $\A$ (i.e., $h_t$ is a strict control),
\beqs
\mu_{t + 1} = \Phi(\mu_t, \delta_{\alpha_t}) & =& \int_{s \in \Sc} \mu_t(ds)P(s, \mu_t, \alpha_t(s),  {\color{black}\alpha_t \star \mu_t})), \;\; t \in \N \cup \{0\},
\enqs
where $\alpha_t \star \mu_t$ is the pushforward of measure $\mu_t$. 
\end{Lemma}

\noindent \begin{proof}{Proof of Lemma \ref{lem: flowmut}.} \; Fix ${\pi} =\{\pi_t\}_{t=0}^\infty$ $\in$ $\Pi$. For any bounded measurable function $\varphi$ on $\Sc$, by the law of iterated conditional expectation,
\beqs
\E^{\pi}[\varphi(s_{t + 1})] &=& \E^{\pi}\Big[\E^{\pi}\big[\varphi(s_{t + 1}) \big| s_1 \cdots, s_t\big]\Big]\\
&=& \E^{\pi}\Big[\int_{s' \in \Sc} \varphi(s') P(s_t, \mu_t, a_t, \nu_t)(ds')\Big]\\
&=& \int_{s' \in \Sc} \varphi(s') \E^{\pi}\Big[ P(s_t, \mu_t, a_t, \nu_t)(ds')\Big]\\
&=& \int_{s' \in \Sc} \varphi(s') \int_{s \in \Sc} \mu_t(ds)\int_{a \in \A}{\pi}_t(s, \mu_t)(da)P(s, \mu_t, a, \nu(\mu_t, \pi_t(\mu_t)))(ds').
\enqs
\end{proof}

Now,  given Outstanding Assumption (A), adopting the technique from \cite{PW2016} for strict controls, we  can show that  the value function $V^{\pi}(\xi)$ for relaxed controls can still be rewritten in terms of the state distribution flow $\{\mu_t\}_{t=0}^\infty$ and that it depends on the initial random variable $\xi$ only through the probability distribution $\mu$. In other words, $V^\pi(\xi)$ can be written as $v^\pi(\mu)$ for some function $v^\pi: \Pc(\Sc) \to \R$. More precisely, 
\begin{Lemma} \label{lemma1} (Law-invariant property) Under Outstanding Assumption (A), given any $\pi$ $\in$ $\Pi$,  $V^{\pi}(\xi)$ in \reff{eq:vpidef} can be written as
\beq \label{equ:reformv}
v^{\pi}(\mu) &=& \sum_{t=0}^{\infty} \gamma^{t} \widehat r(\mu_t, \pi_t(\mu_t)),
\enq
where the integrated averaged reward function $\widehat r$ is the measurable function from $\Pc(\Sc)$ $\times$ $\Hc$ to $\R$ such that
\beq \label{hatr}
\widehat r(\mu, h) &:=& \int_{s \in \Sc} \mu(ds)\int_{a \in \A} h(s)(da) \int_{r \in \R} r R(s, \mu, a, \nu(\mu, h))(dr).
\enq
In particular, if $h_t = \delta_{\alpha_t}$ for some $\alpha_t:$ $\Sc $ $\to$ $\A$, $t \in \N \cup \{0\}$ (i.e.,  $\pi_t$ is a strict control), and $R(s, \mu, a, \nu)(\cdot) = \delta_{r(s, \mu, a, \nu)} (\cdot)$ for some $r: \Sc \times \Pc(\Sc) \times \A \times \Pc(\A) \to \R$, then with a slight abuse of notation, we can write $v^\pi(\mu) = v^\alpha(\mu)$ and 
\beqs
 v^\alpha(\mu) = \sum_{t=0}^\infty \gamma^t \widehat r(\mu_t, \delta_{\alpha_t}) = \sum_{t=0}^\infty \gamma^t\int_{s \in \Sc} r(s, \mu_t, \alpha_t(s), {\color{black}\alpha_t \star \mu_t})\mu_t(ds).
\enqs


\end{Lemma}
\medskip

The flow property of $\{\mu_t\}_{t=0}^\infty$ {and the law-invariant property of $v^\pi$} in the above lemmas suggest that MFCs with learning may be viewed as an RL
problem with the  state variable $s_t$ replaced by the probability distribution $\mu_t$. 
This view is useful for subsequent analysis, and in particular
critical for establishing the DPP, the main result of the paper.

\section{DPP for Learning MFCs}

\subsection{Time Inconsistency: An Example} 
\label{sec:wrong_Q}
Recall from Section \ref{secsclassicalRL}  the Bellman equation for the Q function
in classical single-agent RL, 
\beqs
Q(s,a)=\E[r(s,a)]+\E_{s'\sim P(s,a)} {\sup_{a' \in \A} Q(s', a')}.
\enqs
It is tempting to define such Q function for MFC in the learning framework.
Unfortunately, such Q function will not be time consistent, as demonstrated in the following example.

\begin{Example} \label{wrong_Q}
Take a two-state dynamic system with two choices of  actions. Denote the state space as {$\Sc$ $=$ $\{L, H\}$} and the action space as {$\A= \{{\rm ST}, {\rm MV}\}$}. The transition probability goes as follows: 
\beqs 
P(s, a, s') = \lambda_s {{\bf{1}}_{\{a = {\rm MV}\}}}, \;\; \text{if}\; s' \neq s \in \Sc, \; P(s, a, s') = 1 - \lambda_s {{\bf{1}}_{\{a = {\rm MV}\}}}, \;\;\text{if}\; s' = s \in \Sc
\enqs
{with $\lambda_{s}\in[0,1]$ for $s\in \mathcal{S}$.}
Here $P(s, a, s')$ is the probability of moving to state $s'$ when the agent in state $s$ takes the action $a$. That is, when the agent in the state $s$ takes action {{\rm ST}}, she will {stay} at the current state $s$; when the agent in the state $s$ takes the action {{\rm MV}}, she will {move} to a different state $s'$ with probability $0\leq \lambda_s \leq 1$, $s \in \Sc$ and stay at state $s$ with probability $1 - \lambda_s$, $s \in \Sc$. After each action, the representative agent will receive a reward  $r_t=$ ${\bf 1}_{\{s_t = H\}} - \big(\E[{\bf 1}_{\{s_t = H\}}]\big)^2 - \lambda \Wc_1(\mu_t, B)$. Here $\mu_t$ denotes the probability distribution of the state at time $t$, $B$ is a given Binomial distribution with parameter $p$ ($1 - \lambda_{L} \le p\le \lambda_{{H}}$), and $\lambda > 0$ is a scalar parameter.  Fix any arbitrary initial state distribution $\mu_0$ $=$ $p_0 {\bf 1}_{\{s_0=L\}} + (1-p_0) {\bf 1}_{\{s_0 =H\}}$ with some $0 \leq p_0 \leq 1$. 
If taking the standard  Q function with  the state variable  $s$ and the action variable $a$ instead of the state distribution $\mu$ and the local policy $h$, this leads to the standard Q-learning update:
 \begin{eqnarray} \label{wrong_bellman}
Q_{t + 1}(s_t, a_t) = (1 - l_t) Q_t(s_t, a_t) + l_t \times \big(r_t +  \gamma  \times \max_{a' \in \A}(Q_t(s_{t + 1}, a')\big).
\end{eqnarray}
Here $a_t \in \A$  is the action from all agents in the state $s_t$ at $0 \leq t \leq T$, ${l_t}$ is the learning rate of Q table, and  $r_t$ is the observed reward sampled from taking action $a_t$. Suppose that agents in both states ({{$L$}} and {{$H$}}) will choose actions according to an $\epsilon$-greedy policy. Namely, in each iteration $t$,  each agent in state $s$ ($s=L$ or $s=H$) will choose an action from $\arg\max_{a\in \mathcal{A}} Q_t(s,a)$ with probability $1-\epsilon$ and choose an arbitrary action with probability $\epsilon$. Then $\mu_t$ evolves according to Equation \reff{eqv:flowmut} with any initial population distribution $\mu_0$ under this $\epsilon$-greedy policy.

\end{Example}


{For simplicity, the Q function is initialized to be zero for every $s \in \Sc, a \in \A$. } Following this Q-leaning update \reff{wrong_bellman}, the experiment result on the convergence of Q function is reported below, with $T=10000$, $p=0.6$,  $\lambda_{{L}} = 0.5$, $\lambda_{{H}} =0.8$, $\lambda=10$,  $\gamma=0.5$.  Following \cite{even2003learning}, the learning rate is set as $l_t = \frac{1}{\# (s_t,a_t)+1}$ with $\# (s_t,a_t)$ the number of total visits to state-action pair $ (s_t,a_t)$ up to iteration $t$.

\begin{table}[H]
\centering
\caption{Convergence of Q function with different initial distribution.}\label{bad_Q_table}
\begin{tabular}{ |p{3.1cm}||p{2.6cm}|p{2.6cm}|p{2.6cm}| p{2.6cm} |}
 \hline
 \hline
 Initial distribution& $Q_T(L,{\rm ST})$&$Q_T(L,{\rm MV})$&$Q_T(H,{\rm ST})$&$Q_T(H,{\rm MV})$\\
 \hline
$p_0=0.01$  &-4.41       & -4.41 &-3.24 & -3.58\\
$p_0=0.5$ &  -4.56    & -4.36 & -3.45 & -3.45 \\
$p_0=0.99$  &-4.87& -4.69 & -3.78& -3.78 \\\hline
\end{tabular}
\end{table}

Note the time inconsistency here: with different initial population distribution $\mu_0$, Q table will converge to different values. The culprit:  with this form of Q function, the state space and the action space are not rich enough to ensure the DPP or the Bellman equation for \eqref{wrong_bellman}. 

\subsection{{IQ} function for MFCs with learning}

Example \ref{wrong_Q}  indicates the wrong form of the Q function for MFCs with learning.
Therefore, our first step is to define an appropriate Q function.
The question is, what is wrong with the previous one?

First, recall that
MFC as a cooperative game is essentially an auxiliary control problem:
instead of maximizing reward for each individual agent, the objective in MFC  is to maximize the collective reward from  the perspective of the  auxiliary controller.  The  auxiliary controller's  value function  depends  on the probability distribution of the state  $\mu$. Therefore, the  Q function for MFCs  should  be dependent on $\mu$ instead of $s$.

Secondly, Lemma \ref{lem: flowmut} suggests that  once a control ${\pi} \in \Pi$ is given, the dynamics of the state distribution is determined by $\mu_{t+1}=\Phi(\mu_t,h_t)$, which is a {\it deterministic process} through $h_t$ in $\Pc(\Sc)$. Therefore,  an appropriate  Q function should be a  function on  $\Hc$, rather than of the single action in $\A$ or a probability distribution on $\A$. In other words,  the learning problem for MFCs should be recast as control problems with the probability measure space as the new state-action space such that 
\begin{equation}\label{mfc_objective_1}
v(\mu) := \sup_{\pi\in\Pi}\sum_{t=0}^{\infty} \gamma^{t} \E\big[\widehat r(\mu_t, \pi_t(\mu_t))\big]
\end{equation}
\vspace{-0.3em}
subject to
\vspace{-0.3em}
\begin{equation}\label{mfc_dynamics_2}
\mu_{t + 1}=\Phi(\mu_t, \pi_t(\mu_t)), \; t \in \N \cup \{0\}, \;\;\; \mu_0=\mu.
\end{equation}

Accordingly,  the appropriate Q function for  MFCs with learning should be defined by 
``lifting'' the classical Q function in RL, with lifting in the sense of replacing the
 state space $\Sc$ and action space $\A$ by the
 state space $\Pc(\Sc)$ and action space $\Hc$ respectively.
Hence, the proper Q function for MFCs with learning should take the following {\it integral} form, called, integrated Q (IQ) function.

\begin{Definition}[\underline{IQ Function}]
Given the framework of MFC with learning in \reff{eq:vpidef}-\reff{equ: defv}, the {IQ} function is a measurable real-valued function defined on $\Pc(\Sc) \times \Hc$ such that
\beq\label{eq:Q-MKV} 
Q(\mu, h) &=& \sup_{{\pi}\in\Pi_1} Q^{{\pi}}(\mu, h), \;\;\; \text{for any}\; \mu \in \Pc(\Sc), \; h \in \Hc,
\enq
\vspace{-1em}
with
\vspace{-1em}
\beqs
Q^{{\pi}}(\mu, h) &=& \E^{\pi}\biggl[\sum_{t=0}^\infty \gamma^{t} r(s_t, \mu_t, a_t, \nu_t) \biggl| s_0\sim\mu, a_0\sim h(s_0), a_t \sim \pi_t(s_t, \mu_t), t \in \N \biggl].
\enqs
\end{Definition}

\subsection{DPP: Necessary for {IQ} Function}

The above specification of the {IQ} function 
enables us to  establish the  DPP for MFCs with learning, 
in the form of the following Bellman equation.
\begin{Theorem}(DPP for {IQ} function) \label{thm: BellmanQ}
Under Outstanding Assumption (A), for any $\mu$ $\in$ $\Pc(\Sc)$ and $h$ $\in$ $\Hc$,
\beq \label{equ:BellmanQ}
Q(\mu, h) &=& \widehat r(\mu, h) + \gamma \sup_{h'\in \Hc} Q(\Phi(\mu, h), h').
\enq
\end{Theorem}
 {The idea for the proof of Theorem \ref{thm: BellmanQ} is borrowed from Theorem 3.1 in \cite{PW2016}. Unlike \cite{PW2016}, which considers the value function for MFC with strict controls, we consider the IQ function for MFC with learning over an infinite time horizon with a stochastic reward function and with relaxed controls. For completeness, we highlight the key step for the proof of Theorem \ref{thm: BellmanQ}.}

\noindent \begin{proof}{Proof of Theorem \ref{thm: BellmanQ}.}
To start,  fix some arbitrary $\mu$ $\in$ $\Pc(\Sc)$ and ${\pi}$ $\in$ $\Pi$, we have
\beq \label{DPPvpi}
v^{\pi}(\mu) &=& \widehat r(\mu, \pi_0(\mu)) + \sum_{t=1}^\infty \gamma^t \widehat r(\mu_t, \pi_t(\mu_t)) \nonumber\\
&=& \widehat r(\mu, \pi_0(\mu)) + \gamma v^{\pi_{-}}(\Phi(\mu, \pi_0(\mu))), \nonumber\\
&=& Q^{\pi_{-}}(\mu, \pi_0(\mu)),
\enq
where ${\pi}_{-}$ $:=$ $\{\pi_t\}_{t=1}^\infty$ $\in \Pi_1$, and the second equality uses the flow property of {$\{\mu_t\}_{t=0}^\infty$} from Lemma \ref{lem: flowmut}.

Now we can  establish the following relation between the value function $v$ and the {IQ} function, 
\beq \label{equ: relationvQ}
v(\mu) &=& \sup_{h \in \Hc} Q(\mu, h), \;\;\; \text{for any} \; \mu \in \Pc(\Sc).
\enq
To prove \reff{equ: relationvQ}, we first show $v(\mu) \le \sup_{h \in \Hc} Q(\mu, h)$ for any $\mu \in \Pc(\Sc)$. To see this, note
\beq \label{proofequ: lhs}
v^{\pi}(\mu) &=& Q^{\pi_{-}}(\mu, \pi_0(\mu)) \;\leq\; Q(\mu, \pi_0(\mu))  \; \leq \; \sup_{h \in \Hc} Q(\mu, h),
\enq
where the first inequality is by  definition of $Q(\mu, h)$, and by the fact that $\pi_0(\mu)$ $\in$ $\Hc$ for each $\mu$ $\in$ $\Hc$.
Taking supremum over all policies ${\pi}$ $\in$ $\Pi$ in \reff{proofequ: lhs} shows that
\beq \label{equLHSless}
v(\mu) &\leq & \sup_{h \in \Hc} Q(\mu, h).
\enq
To see for any $\mu \in \Pc(\Sc)$, $v(\mu) \ge  \sup_{h \in \Hc} Q(\mu, h)$, fix any arbitrary $\mu$ $\in$ $\Pc(\Sc)$ and $\pi_0(\mu)$ $\in$ $\Hc$, for any $\epsilon$ $>$ $0$, there exists ${\pi}^\epsilon$ $=$ $\{\pi^\epsilon_t\}_{t=1}^\infty$ $\in \Pi_1$ such that
\beq \label{vpiepsilon}
Q^{{\pi}^\epsilon}(\mu, \pi_0(\mu)) \geq Q(\mu, \pi_0(\mu)) - \epsilon.
\enq
Now define $\tilde {\pi}$ $=$ $\{\tilde\pi_t\}_{t=0}^\infty $ $\in$ $\Pi$ by $\tilde\pi_t = \pi_0 1_{\{t=0\}} + \pi^\epsilon_t 1_{\{t \in \N\}},$
then from \reff{DPPvpi} and \reff{vpiepsilon},
\beq \label{proofequ: rhs}
v(\mu) &\geq& v^{\tilde{\pi}}(\mu)\;=\; Q^{\pi^\epsilon}(\mu, \pi_0(\mu))
\geq Q(\mu, \pi_0(\mu)) - \epsilon.
\enq
Taking supremum over all $\pi_0$ in \reff{proofequ: rhs}, we obtain
\beqs
v(\mu) & \geq & \sup_{\pi_0} Q(\mu, \pi_0(\mu)) - \epsilon \;=\; \sup_{h \in \Hc} Q(\mu, h) -\epsilon.
\enqs
Since the above inequality holds for any $\epsilon$ $>$ $0$,
\beq \label{equLHSmore}
v(\mu) & \geq &  \sup_{h \in \Hc} Q(\mu, h).
\enq
\reff{equ: relationvQ} follows from \reff{equLHSless} and \reff{equLHSmore}.\\
Now we are ready to prove \reff{equ:BellmanQ}.
\beqs
Q(\mu, h) &=& \sup_{\pi \in\Pi_1}\E^{\pi}\biggl[\sum_{t=0}^\infty \gamma^t r(s_t, \mu_t, a_t, \nu_t) \biggl| s_0\sim\mu, a_0\sim h(s_0), a_t \sim \pi_t(s_t, \mu_t), t \in \N\biggl] \nonumber\\
&=& \sup_{\pi \in\Pi_1} \big[\widehat r(\mu, h) + \gamma v^{\pi}(\Phi(\mu, h))\big] \nonumber\\
&=& \widehat r(\mu, h) + \gamma v(\Phi(\mu, h)) \nonumber\\
&=& \widehat r(\mu, h) + \gamma \sup_{h'\in \Hc} Q(\Phi(\mu, h), h'),
\enqs
where the second equality is from the flow property of $\{\mu_t\}_{t=0}^\infty$ in Lemma \ref{lem: flowmut}, the third equality is by the definition of the value function, and the last inequality is from \reff{equ: relationvQ}.
\end{proof}

\subsection{DPP: Sufficient for {IQ} Function}
So far, we have established the necessary condition for the Bellman equation. That is, the {IQ} function satisfies the Bellman equation and is time consistent.
We can further establish that this Bellman equation is sufficient, in the form of the following verification theorem.
\begin{Theorem}(Verification theorem) \label{theoremverification}
\begin{itemize}
\item [(1).] Suppose  $\tilde Q: \Pc(\Sc) \times \Hc \to\R$ satisfies the Bellman equation
 \reff{equ:BellmanQ} for any $(\mu, h)$ $\in$ $\Pc(\Sc)$ $\times$ $\Hc$. Suppose that for every $\mu$ $\in$ $\Pc(\Sc)$, one can also find a stationary control $\pi^*(\mu)$ $\in$ $\Hc$ that achieves $\sup_{h \in \Hc}\tilde Q(\mu, h)$ , then $\pi^*$ is an optimal stationary control of problem \reff{eq:vpidef}-\reff{equ: defv}. 
 \item [(2).] If we further assume that there exists $0\leq r_{\text{max}}<\infty$ such that  the sup norm $||r||_\infty\leq r_{\text{max}}, a.s.$,    then Q defined in \eqref{eq:Q-MKV} is the unique solution in $\{q\in \Mc(\Pc(\Sc) \times \Hc):\|q\|_\infty\leq V_{\max}\}$ for the Bellman equation \eqref{equ:BellmanQ}, with $V_{\max}:=\frac{r_{\text{max}}}{1-\gamma}$.  In this case, the stationary control  $\pi^*(\mu)$ $\in$ $\Hc$ that achieves $\sup_{h \in \Hc}\tilde Q(\mu, h)$ is an optimal stationary control of problem \reff{eq:vpidef}-\reff{equ: defv}. 
 \end{itemize}
\end{Theorem}
{The idea for the proof of Theorem \ref{theoremverification}-(1) is borrowed from the proof for Theorem 3.2 in \cite{PW2016}. Nevertheless, the backward induction argument by \cite{PW2016} for  a finite-time-horizon case needs appropriate modification for the IQ function over an infinite time horizon.
We hence highlight the key step here.}

\noindent \begin{proof}{Proof of Theorem \ref{theoremverification}.} (1) {
On  one hand, given any $\mu$ $\in$ $\Pc(\Sc)$, for any given control ${\pi}$ $\in$ $\Pi$, the evolution of $\{\mu_t\}_{t=0}^\infty $ is given by \reff{eqv:flowmut}. From \reff{equ:BellmanQ}
\begin{eqnarray}
\tilde Q (\mu_t, \pi_t(\mu_t)) &\geq& \widehat r(\mu_t, \pi_t(\mu_t)) + \gamma \tilde Q(\mu_{t +1}, \pi_{t + 1}(\mu_{t + 1})), \;\;\; t \in \N \cup \{0\}.\label{eq:oneside}
\end{eqnarray}
Multiplying \eqref{eq:oneside} by $\gamma^{t}$ 
and summing over $0$ $\leq$ $t$ $\leq$ $T-1$ for any fixed $T$, we obtain
\beqs
\tilde Q(\mu, \pi_0(\mu)) -\gamma^T \tilde Q(\mu_T, \pi_T(\mu_T)) &\geq & \sum_{t=0}^{T-1} \gamma^t \widehat r(\mu_t, \pi_t(\mu_t)).
\enqs
As $\lim_{T \to \infty} \gamma^T \tilde Q(\mu, h)$ $=$ $0$ for any fixed $(\mu, h) \in \Pc(\Sc) \times \Hc$, by taking the limit $T$ $\to$ $\infty$,
$
\tilde Q(\mu, \pi_0(\mu)) \geq \sum_{t=0}^\infty \gamma^t \widehat r(\mu_t, \pi_t(\mu_t)) = v^{\pi}(\mu),$ which leads to
 $\sup_{h \in \Hc} \tilde Q(\mu, h) \geq v(\mu).$

On the other hand, since $\pi^*(\mu)$ $\in$ $\arg\max \tilde Q(\mu, h)$ holds for every $\mu$ $\in$ $\Pc(\Sc)$, then 
\beqs
\tilde Q(\mu_t, \pi^*(\mu_t)) &=& \widehat r(\mu_t, \pi^*(\mu_t)) + \gamma \tilde Q(\mu_{t + 1}, \pi^*(\mu_{t + 1})).
\enqs
Repeat the same argument for  $\pi^*$ as for $\pi$, $\sup_{h \in \Hc} \tilde Q(\mu, h) = \tilde Q(\mu, \pi^*(\mu)) = v^{\pi^*}(\mu)$,
which shows that $\pi^*$ is an optimal stationary control.}

(2) First, since $||r||_\infty\leq r_{\text{max}}$ a.s., for any $\mu\in\Pc(\Sc)$ and $h\in\Hc$, the aggregated reward function \reff{hatr} satisfies 
$$|\widehat r(\mu, h)|\leq r_\text{max}\cdot\int_{s \in \Sc} \mu(ds)\int_{a \in \A} h(s)(da)=r_\text{max}.$$
In this case, for any $\mu\in\Pc(\Sc)$ and $h\in\Hc$, $|Q(\mu, h)|\leq r_{\text{max}}\cdot\sum_{t=0}^\infty \gamma^t=V_{\max}$. Hence, $Q\in\{q\in \Mc(\Pc(\Sc) \times \Hc):\|q\|_\infty\leq V_{\max}\}$ and  it satisfies the Bellman equation \eqref{equ:BellmanQ}.
    
    To see that Q is the unique function in $\{q\in \Mc(\Pc(\Sc) \times \Hc):\|q\|_\infty\leq V_{\max}\}$ satisfying \eqref{equ:BellmanQ}, consider the Bellman operator $B:\Mc(\Pc(\Sc) \times \Hc) \to \Mc(\Pc(\Sc) \times \Hc)$ defined by 
\beq \label{Bellmanoperator}
({B}\,q)(\mu, h)={\widehat r}(\mu, h) + \gamma \sup_{\tilde{h}\in\Hc}q({\Phi}(\mu, h),\tilde{h}).
\enq
Then $B$ is a contraction operator on $\{q\in \Mc(\Pc(\Sc) \times \Hc):\|q\|_\infty\leq V_{\max}\}$:  clearly $B$ maps $\{q\in \Mc(\Pc(\Sc) \times \Hc):\|q\|_\infty\leq V_{\max}\}$ to itself, and for any $(\mu,h)\in \Pc(\Sc) \times \Hc$,
    \begin{align*}
        |B q_1(\mu,h)-B q_2(\mu,h)| \leq \gamma \sup_{\tilde{h}\in\Hc} |q_1(\Phi(\mu,h),\tilde{h})-q_2(\Phi(\mu,h),\tilde{h})| \leq \gamma \|q_1-q_2\|_\infty.
    \end{align*}
    Thus, $\|B q_1-B q_2\|_\infty\leq\gamma \|q_1-q_2\|_\infty$. Therefore, $B$ is a contraction mapping with modulus $\gamma<1$ under the sup norm on $\{q\in \Mc(\Pc(\Sc) \times \Hc):\|q\|_\infty\leq V_{\max}\}$. Hence the uniqueness by {{the contraction property}}. 
\end{proof}
\subsection{{IQ} function vs classical Q function.}\label{I-remark}

 Comparing {IQ} function  and the classical Q function, there is an analytical connection between their respective Bellman equations.
 
  To see this,  consider the simplest problem of MFCs with learning where there are no state distribution nor action distribution in the probability transition function $P$ or in the deterministic reward function $r$.   Assume $\Sc$ and $\A$ are finite so that there exists $ r_{\text{max}}>0$ such that $\|r\|_\infty \leq  r_{\text{max}}$. Here for clarity, we shall distinguish the classical single-agent Q function \reff{defsingleQ} and the {IQ} function \reff{eq:Q-MKV} by writing  $Q_{\rm single}$ and $Q_{\rm mfc}$ respectively. 
  Then we see that the {IQ} function in \reff{eq:Q-MKV} is the integral of Q function in \reff{defsingleQ} such that 
\beq \label{equ:standardQ}
Q_{\rm mfc}(\mu, h) &=& \sum_{s \in \Sc} \mu(s)\sum_{a \in \A}  Q_{\rm single}(s, a)h(s)(a).
\enq

To see this connection, define  
$$\tilde Q(\mu, h)=\sum_{s \in \Sc} \mu(s)\sum_{a \in \A}  Q_{\rm single}(s,a)h(s)(a).$$ 
Note that $\tilde Q$ is linear in $\mu$ and $h$. From the Bellman equation \reff{equ: singleQ} of $Q_{\rm single}$ \reff{defsingleQ}, we have
\beqs
\tilde Q(\mu, h) &=&  \widehat r(\mu, h) + \gamma \sum_{s \in \Sc}\mu(s)\sum_{a \in \A}h(s)(a)\sum_{s' \in \Sc} P(s, a)(s')\max_{a' \in \A}  Q_{\rm single}(s', a'),
\enqs
then we can see that
\beqs \label{equ: Qphi}
 \sum_{s \in \Sc}\mu(s)\sum_{a \in \A}h(s)(a)\sum_{s' \in \Sc} P(s, a)(s')\max_{a' \in \A}  Q_{\rm single}(s', a') &=& \sup_{h' \in \Hc}\tilde Q(\Phi(\mu, h), h').
\enqs
In fact, on one hand, for any $h'$ $\in$ $\Hc$,
\beqs
& &\sum_{s \in \Sc} \mu(s) \sum_{a \in \A} h(s)(a) \sum_{s' \in \Sc} P(s, a)(s') \max_{a' \in \A}  Q_{\rm single}(s', a')\\
= & & \sum_{s \in \Sc}\mu(s) \sum_{a \in \A} h(s)(a) \sum_{s' \in \Sc} P(s, a)(s') \sum_{\tilde a \in \A} h'(s')(\tilde a) \max_{a' \in \A} Q_{\rm single}(s', a')\\
\geq & & \sum_{s \in \Sc} \mu(s)\sum_{a \in \A} h(s)(a)\sum_{s' \in \Sc} P(s, a)(s') \sum_{\tilde a \in \A} h'(s')(\tilde a)  Q_{\rm single}(s', \tilde a)\\
= & & \sum_{s' \in \Sc} \Phi(\mu, h)(s') \sum_{\tilde a \in \A} h'(s')(\tilde a) Q_{\rm single}(s', \tilde a)\\
= & & \tilde Q(\Phi(\mu, h), h'),
\enqs
where the first equality is from $\sum_{\tilde a \in \A} h(s')(\tilde a)$ $=$ $1$, the second equality is by \reff{equ: Phi}, and the last equality is by the definition of $\tilde Q$.\\
On the other hand, if we take
\beqs
h'_*(s') &=& \left\{\begin{array}{ccl}
\frac{1}{\# \arg\max_{a' \in \A}  Q_{\rm single}(s', a')}, &\quad \mbox{if} \; a_*(s') \in \arg\max_{a' \in \A}  Q_{\rm single}(s', a')\\
0, & \mbox{otherwise},
\end{array}\right.
\enqs
with $\# \arg\max_{a' \in \A}  Q_{\rm single}(s', a')$ the number of elements in $\arg\max_{a' \in \A}  Q_{\rm single}(s', a')$,
then
\beqs
& & \sup_{h' \in \Hc} \tilde Q(\Phi(\mu, h), h') \geq \tilde Q(\Phi(\mu, h), h'_*)\\
= & & \sum_{s' \in \Sc} \Phi(\mu, h)(s') \sum_{a' \in \A} Q_{\rm single}(s', a') h'_*(s')(a')\\
= & & \sum_{s' \in \Sc} \Phi(\mu, h)(s') \max_{a' \in \A} Q_{\rm single}(s', a')\\
= & & \sum_{s \in \Sc}\mu(s)\sum_{a \in \A} h(s)(a) \sum_{s' \in \Sc} P(s, a)(s')\max_{a' \in \A} Q_{\rm single}(s', a').
\enqs
Therefore,
\beqs
\tilde Q(\mu, h) &=& \widehat r(\mu, h) +\gamma \sup_{h' \in \Hc} \tilde Q(\Phi(\mu, h), h').
\enqs
Since both $\tilde Q$ and $Q_{\rm mfc}$ satisfy Bellman equations \reff{equ:BellmanQ}, we have $Q_{\rm mfc}$ $=$ $\tilde Q$ from the uniqueness of the fixed point of a contraction mapping $B$ in \reff{Bellmanoperator}.

\begin{Remark}
The relationship between $Q_{\rm mfc}$ and $Q_{\rm single}$ in \eqref{equ:standardQ} is intriguing for algorithmic designs: the ``global'' Q table $(Q_{\rm mfc})$ needs to be trained in a centralized manner by observing the population state distribution; yet agents only need to maintain a ``local'' Q table $(Q_{\rm single})$ for execution. 
\end{Remark}


\subsection{DPP for the Value Function and Value-iteration Algorithms}
For  model-based learning algorithms such as the  value iteration, we have the Bellman equation for the value function $v$ from Theorem \ref{thm: BellmanQ}.

\begin{Theorem} (DPP for value function) \label{cor: Bellmanv} Under Outstanding Assumption (A), the value function $v$ satisfies the Bellman equation
\beq \label{equ:Bellmanv1st}
v(\mu) &=& \sup_{h \in \Hc} \big\lbrace\widehat r(\mu, h) + \gamma v(\Phi(\mu, h))\big\rbrace, \;\;\; \text{for any} \; \mu \in \Pc(\Sc).
\enq
\end{Theorem}

\medskip

Given $\pi: \Pc(\Sc) \to \Hc$, define the operator $T_\pi:$ $\Mc(\Pc(\Sc))$ $\to$ $\Mc(\Pc(\Sc))$ such that
\beq
\label{Tpiv}
(T_\pi w)(\mu) &:=& \widehat r (\mu, \pi(\mu)) + \gamma w (\Phi(\mu, \pi(\mu))),
\enq
and another operator $T:$ $\Mc(\Pc(\Sc))$ $\to$ $\Mc(\Pc(\Sc))$ such that
\beq\label{Tv}
(T w)(\mu) &:=& \sup_{h \in \Hc}\big\lbrace\widehat r(\mu, h) + \gamma w(\Phi(\mu, h))\big\rbrace,
\enq
where $\widehat r(\mu, h)$ and $\Phi(\mu, h)$ are given in \reff{hatr} and \reff{equ: Phi}.
\begin{Proposition} \label{prop: valueiteration}
 Assume without loss of generality $v_0$ $=$ $0$, then under Outstanding Assumption (A), we have for all $\mu$ $\in$ $\Pc(\Sc)$,
$$v(\mu)=\lim_{n \to \infty} (T^n v_0)(\mu),$$ where $T^n$ is $n$-th composition of $T$ such that  such that $T^n = \underbrace{T \circ \cdots \circ T}_{n}$.
\end{Proposition}
Proof of Proposition \ref{prop: valueiteration} relies on the following Lemma.

\begin{Lemma} \label{lemma: value_iteration} Assume Outstanding Assumption (A) and without loss of generality $v_0$ $=$ $0$, for any $\mu$ $\in$ $\Pc(\Sc)$ and ${\pi}$ $=$ $\{\pi_t\}_{t=0}^n$ with $\pi_t: \Pc(\Sc) \to \Hc$ for every  $0$ $\leq$ $t$ $\leq$ $n$,
\beq  \label{equ: Tpinv0}
( T_{\pi_{0}}\cdots T_{\pi_n} v_0)(\mu) &=& \sum_{t=0}^n \gamma^t \widehat r(\mu_t, \pi_t(\mu_t)),\\
(T^{n+1} v_0)(\mu) &=& \sup_{\{\pi_t\}_{t=0}^n} (T_{\pi_0} \cdots T_{\pi_n} v_0)(\mu),
\label{equ: TTpirelation}
\enq
where $T_{\pi_0} \cdots T_{\pi_n}$ is the composition of all $T_{\pi_t}$, $0 \leq t \leq n$.
\end{Lemma}
\begin{proof}{Proof of Lemma \ref{lemma: value_iteration}.}
We prove  \reff{equ: TTpirelation} (and similarly \reff{equ: Tpinv0}) by the forward induction. 
The result clearly holds for $n$ $=$ $0$ as
\beqs
\sup_{\pi_0}(T_{\pi_0}v_0)(\mu) = \sup_{\pi_0} \widehat r(\mu, \pi_0(\mu)) = \sup_{h \in \Hc}\E[\widehat r(\mu, h)] = (T v_0)(\mu).
\enqs
Suppose that \reff{equ: TTpirelation} holds for $n$ $=$ $k$, then for $n$ $=$ $k + 1$,
\beq\label{TTpikrelation1k}
(T^{k + 1} v_0)(\mu) &=& \sup_{h \in \Hc} \big\lbrace\widehat r(\mu, h) + \gamma (T^k v_0)(\Phi(\mu, h))\big\rbrace \nonumber\\
&=&  \sup_{h \in \Hc}\big\lbrace\widehat r(\mu, h) + \gamma\sup_{\{\tilde\pi_t\}_{t=0}^{k-1}} (T_{\tilde\pi_0} \cdots T_{\tilde \pi_{k-1}} v_0)(\Phi(\mu, h))\big\rbrace \nonumber\\
&=& \sup_{h \in \Hc} \big\lbrace\widehat r(\mu, h) + \gamma \sup_{\{\pi_{t}\}_{t=1}^{k}} (T_{\pi_1} \cdots T_{\pi_{k}} v_0)(\Phi(\mu, h))\big\rbrace \nonumber\\
&=& \sup_{h \in \Hc,\{\pi_{t}\}_{t=1}^{k }} \big\lbrace\widehat r(\mu, h) + \gamma (T_{\pi_1}\cdots T_{\pi_{k}} v_0)(\Phi(\mu, h))\big\rbrace \nonumber\\
&=& \sup_{\{\pi_t\}_{t =0}^{k + 1}} (T_{\pi_0}\cdots T_{\pi_{k}} v_0)(\mu), \nonumber
\enq
where the first equality is from the definition of $T$ in \reff{Tv}; the second equality is by the assumption that \reff{equ: TTpirelation} holds for $n$ $=$ $k$.
\end{proof}

\medskip

\noindent  \begin{proof}{Proof of Proposition \ref{prop: valueiteration}.} Rewrite $v^{\pi}(\mu)$ as
\beq \label{equ: vpi2part}
v^{\pi}(\mu) &=& \sum_{t=0}^{n-1} \gamma^t \widehat r(\mu_t, \pi_t(\mu_t)) + \sum_{t= n}^\infty \gamma^t \widehat r(\mu_t, \pi_t(\mu_t)) \nonumber\\
&=& (T_{\pi_0} \cdots T_{\pi_{n-1}} v_0)(\mu) + \sum_{t=n}^\infty \gamma^t \widehat r(\mu_t, \pi_t(\mu_t)),
\enq
where the second equality is by \reff{equ: Tpinv0}. Now Outstanding Assumption (A) implies $$\lim_{n \to \infty}\sup_{{\pi}}\sum_{t=n}^\infty \gamma^t |\widehat r(\mu_t, \pi_t(\mu_t))|=0.$$
Taking supremum  over ${\pi}$ $\in$ $\Pi$ in \reff{equ: vpi2part} together with \reff{equ: TTpirelation} gives
\beqs \label{equ: v2part<}
v(\mu) &\leq& (T^{n } v_0)(\mu) + \sup_{{\pi}} \sum_{t=n}^\infty \gamma^t|\widehat r(\mu_t, \pi_t(\mu_t))|,\\
v(\mu) & \geq & (T^{n } v_0)(\mu) -  \sup_{{\pi}} \sum_{t=n}^\infty \gamma^t |\widehat r(\mu_t, \pi_t(\mu_t))|.
\label{equ: v2part>}
\enqs
Taking the limit as $n$ $\to$ $\infty$ together with \reff{equ: TTpirelation} yields $v(\mu) =\lim_{n \to \infty} (T^n v_0)(\mu)$.
\end{proof}


\section{Example \ref{wrong_Q} revisited} \label{sec:example_revisit}

\begin{Example}
Take a two-state dynamic system with two choices of  actions. The state space  {$\Sc$ $=$ $\{L, H\}$} and the action space {$\A= \{{\rm ST}, {\rm MV}\}$}. The transition probability goes as follows: 
\beqs 
P(s, a)(s') = \lambda_s {{\bf{1}}_{\{a = {\rm MV}\}}}, \;\; \text{if}\; s' \neq s \in \Sc, \; P(s, a)(s') = 1 - \lambda_s {{\bf{1}}_{\{a = {\rm MV}\}}}, \;\;\text{if}\; s' = s \in \Sc
\enqs
{with $\lambda_{s}\in[0,1]$ for $s\in \mathcal{S}$.}
Here $P(s, a)(s')$ is the probability of moving to state $s'$ when the agent in state $s$ takes the action $a$; 
 when the agent in the state $s$ takes action {{\rm ST}}, she will {stay} at the current state $s$; when the agent in the state $s$ takes the action {{\rm MV}}, she will {move} to a different state $s'$ with probability $0\leq \lambda_s \leq 1$, $s \in \Sc$ and stay at state $s$ with probability $1 - \lambda_s$, $s \in \Sc$. After each action, the representative agent will receive a reward  $r_t=$ ${\bf 1}_{\{s_t = H\}} - \big(\E[{\bf 1}_{\{s_t = H\}}]\big)^2 - \lambda \Wc_1(\mu_t, B)$.
 Here $\mu_t$ denotes the probability distribution of the state at time $t$, $B$ is a given Binomial distribution with parameter $p$ ($1 - \lambda_L \le p\le \lambda_H$), and $\lambda > 0$ is a scalar parameter.
Fix any arbitrary initial state distribution $\mu_0$ $=$ $p_0 1_{\{s_0=L\}} + (1-p_0) 1_{\{s_0 =H\}}$ for some $0 \leq p_0 \leq 1$.\\
Note that the expected value of immediate reward $\E[r_t]$ at each time $t$ is 
\beqs
\E[r_t] = \E[{\bf 1}_{\{s_t = H\}}] - \E[{\bf 1}_{\{s_t = H\}}]^2 - \lambda \Wc_1(\mu_t, B) = \mu_t(H) - \mu_t(H)^2 - 2\lambda |\mu_t(H) - (1 - p)|,
\enqs 
where $\mu_t(L)$ and $\mu_t(H)$ are the population distribution on state $L$ and $H$ at time $t$, respectively.
Suppose that $\lambda > 0$ is large enough, we have $\max_{\pi}\big(\E[{\bf 1}_{\{s_t=H\}}] - \E[{\bf 1}_{\{s_t=H\}}]^2 - \lambda \Wc_1(\mu_t, B)\big) = 1- p - (1-p)^2$ when $\mu_t = B$ for any $t \in \N$. Therefore, the value function is optimal if and only if the population distribution $\{\mu_t^*\}_{t=1}^\infty$ corresponding to the optimal control $\pi^*$ is given by
\beqs
\mu_t^* \;=\; B = p {\bf 1}_{\{s=L\}}+ (1-p) {\bf 1}_{\{s=H\}}, \;\; t \in \N, \;\;\;\;\mu_0 \;=\; p_0 {\bf 1}_{\{s=L\}} + (1- p_0) {\bf 1}_{\{s=H\}}.
\enqs
From the flow property of $\{\mu_t^*\}_{t=1}^\infty$ in \reff{eqv:flowmut} and \reff{equ: Phi}, we get
\beqs
\mu_{1}^*(L) = p \;=\; \Phi(\mu_0, \pi^*)(L)  &=& \sum_{s \in \Sc}\mu_0(s)\sum_{a \in \A}  P(s, a)(L) \pi^*(s)(a),\\
\mu_{t + 1}^*(L) = p \;=\; \Phi(\mu_t^*, \pi^*)(L) &=& \sum_{s \in \Sc}\mu_t^*({s}) \sum_{a \in \A} P(s, a)(L)\pi^*(s)(a),\; t \in \N,
\enqs
which gives the optimal control and the optimal value
\begin{eqnarray}\label{stationary_policy}
& \pi^*(L) \;=\; (1- \frac{1-p}{\lambda_L}) {\bf 1}_{\{a= {\rm ST}\}} + \frac{1-p}{\lambda_L}{\bf 1}_{\{a = {\rm MV}\}},\;\pi^*(H) \;=\; (1-\frac{p}{\lambda_H}) {\bf 1}_{\{a = {\rm ST}\}} + \frac{p}{\lambda_H} {\bf 1}_{\{a = {
\rm MV}\}}, \\
&v^{\pi^*}({\mu_0}) = 1- p_0 - (1 - p_0)^2 -2\lambda |p_0 - p| + \frac{\gamma}{1- \gamma} \Big(1 - p - (1 - p)^2\Big). \label{eqn: optimal_value}
\end{eqnarray} 
Now, the Q-learning update at each iteration
$t$ using the {IQ} function is
\begin{eqnarray} \label{eqn:IQ_update}
Q_{t + 1}(\mu, h) &=& Q_t(\mu, h) + {l_t} \times \left(\hat{r}_t + \gamma  \sup_{h' \in \Hc}Q_t(\Phi(\mu, h), h') - Q_t(\mu, h) \right).
\end{eqnarray}
Here {$l_t$} is the learning rate {at iteration $t$} and  $\gamma$ is  the discount factor. 
\end{Example}
 This example is  further studied by \cite{MP2019} later on.

Next, we design a simple algorithm (Algorithm~\ref{MFCs_Q}) to show the performance of the {IQ} update \eqref{eqn:IQ_update}, with the following specifications.
We emphasize that the focus here is  the time consistency property not the efficiency of the algorithm. In the experiment, we shall  use  element $(p, 1 - p)$ in the Euclidean space $\R^2$ to denote the Binomial distribution with parameter $p$.

\begin{itemize}
    \item[(a)] {\bf Dimension reduction}: Since $\mu_t(L)+\mu_t(H)=1$  $(t=0,1,\cdots,T)$ , $\pi(L, {\rm ST})+\pi(L, {\rm ST})=1$ and $\pi(H, {\rm ST})+\pi(H, {\rm MV})=1$ for any distribution $\mu_t$ and control $\pi$, we can reduce the dimension of the {IQ} function. If we define $Q(\mu^{L},\pi_{{\rm ST}}^{L},\pi_{{\rm ST}}^{H})$ {and $\Phi(\mu^L,\pi_{{\rm ST}}^L,\pi_{{\rm ST}}^H)$}, with $\mu^L: =\mu(L)$  the probability of population state $L$, $\pi_{{\rm ST}}^{L}: = \pi(L, {\rm ST})$  the probability of the action to ``stay'' at state $L$, and $\pi_{{\rm ST}}^{H}: = \pi(H, {\rm ST})$  the probability of the action to ``stay'' at state $H$, then  $ Q(\mu, \pi)=Q(\mu^L,\pi_{{\rm ST}}^L,\pi_{{\rm ST}}^H)$, $\Phi(\mu, \pi) = \Phi(\mu^L,\pi_{{\rm ST}}^L,\pi_{{\rm ST}}^H)$ with a slight abuse of notation.
    \item[(b)] {\bf Distribution discretization}:  To examine the time-consistency property of \eqref{equ:BellmanQ}  we discretize the state and action distribution with finite precision and apply the classical Q-learning update to \eqref{eqn:IQ_update_reduction}  with finite-dimensional inputs. For simplicity, we assume uniform discretization such that $\tilde{\Pc}(\A):=\{i/N_a: 0\leq i \leq N_a\}$ and $\tilde{\Pc}(\Sc):=\{i/N_s: 0\leq i \leq N_s\}$ for some constant integers $N_a>0$ and $N_s>0$. (For more refined discretization other than the uniform one, see for example the  $\epsilon$-Net approach  in \cite{GHXZ2019}). 

       \item[(c)] {\bf Algorithmic design}:
  The algorithm is summarized in Algorithm~\ref{MFCs_Q}. Note that \eqref{eqn:IQ_update_reduction} is the reduced form of the original update \eqref{eqn:IQ_update} with a discretized distribution. In order to perform the for-loop (Step $3$, $4$, and $5$) in Algorithm \ref{MFCs_Q}, we assume the accessibility to a {population simulator} $(\mu^{\prime}, \widehat{r}) = \mathcal{G}(\mu, \pi)$. That is, for any pair $(\mu,\pi) \in \Pc(\Sc) \times \Hc$, we can sample the aggregated population reward $\widehat r $ and the next population state distribution $\mu'$ under control $\pi$.
    \item[(d)] {\bf Metric design}: Explicit calculations show that the stationary optimal control is given by \eqref{stationary_policy}. Therefore, we design the following metric to check the convergence of the Q table to the true value $v^{\pi^*}(\mu_0)$ in \reff{eqn: optimal_value} and the speed of the convergence.
    \[
    E(t)=\frac{1}{N_s}\sum_{i=0}^{N_s} \left|Q_{t}\left(\frac{i}{N_s},\mbox{Proj}(\pi^{L, *}_{\rm ST}, \tilde\Pc(\A)), \mbox{Proj}(\pi^{H, *}_{\rm ST}, \tilde\Pc(\A))\right) - v^{\pi^*}\left(\frac{i}{N_s}, 1 - \frac{i}{N_s}\right)\right|.
    \]
    Here for simplicity we take $N_s=N_a$; $\pi^{s, *}_{\rm ST}$, $s \in \Sc$, is the optimal control $\pi^*$ in \reff{stationary_policy} evaluated in state $s$ and action ST; the projection is defined as $\mbox{Proj}(\pi^{s, *}_{\rm ST}, \tilde \Pc(\A)): = \arg\min_{\tilde\pi^s_{\rm ST} \in \tilde\Pc(\A)}\left|\pi^{s, *}_{\rm ST} - \tilde\pi^s_{\rm ST}\right|$.
    \item[(e)] {\bf Parameter set-up}: Parameters are set as follows: $T=20$, $p=0.6$, {a constant learning rate $l_t=l=0.4$ for all $t$}, $\gamma=0.5$, $\lambda_L = 0.5$, $\lambda_R = 0.8$, $\lambda = 10$ and $N_a=N_s=20$. Each component in $Q_0$ is randomly initialized from a uniform distribution on $[0,1]$. The experiments are repeated $20$ times.
    \item[(f)]{\bf Performance analysis.} The experiments show that metric $E(t)$ converges in around $15$ outer iterations (Figure~\ref{fig:errors}). The standard deviation of $20$ repeated experiments is very small. This is partially due to lifting of the state-action space which leads to the deterministic property of the underlying system.

\end{itemize}

Recall $\tilde \Pc(\Sc) =\{{i}/{N_s}: 0\leq i \leq N_s\}$.
Further denote the projection as 
\begin{eqnarray}\label{eq:projection}
\mbox{Proj}({\Phi}(\mu^L,\pi^L_{\rm ST}, \pi_{\rm ST}^H)(L),\tilde \Pc(\Sc)):=\arg\Min_{\tilde \mu^L \in\tilde P(\Sc)}|{\Phi}(\mu^L, \pi_{\rm ST}^L, \pi_{\rm ST}^H)(L)-\tilde{\mu}^L|.
\end{eqnarray}

Then the algorithm is summarized as follows.

\begin{algorithm}[H]
  \caption{\textbf{MFCs Q-learning with distribution discretization}}
  \label{MFCs_Q}
\begin{algorithmic}[1]
  \STATE \textbf{Input}: $N_a$ and $N_s$.
 \STATE {\textbf{Initialization}: $Q_0(\mu^L, \pi_{\rm ST}^L, \pi_{\rm ST}^H) =0$ for every $(\mu^L, \pi_{\rm ST}^L, \pi_{\rm ST}^H) \in \tilde \Pc(\Sc) \times(\tilde \Pc(\A))^2$}
 \FOR {$t= 0,1, \cdots,T-1$}
  \FOR {$\pi_{\rm ST}^L\in\{\frac{i}{N_a}, 0\leq i \leq N_a\}$}
  \FOR {$\pi_{\rm ST}^H\in\{\frac{i}{N_a}, 0\leq i \leq N_a\}$}
   \FOR {$\mu^L\in\{\frac{i}{N_s}, 0\leq i \leq N_s\}$}
 \STATE {$\mu^{L\prime} = \mbox{Proj}({\Phi}(\mu^L,\pi_{\rm ST}^L, \pi_{\rm ST}^H)(L),\tilde \Pc(\Sc))$}
 \STATE  {\begin{equation} \label{eqn:IQ_update_reduction}
Q_{t + 1}(\mu^L, \pi_{\rm ST}^L,\pi_{\rm ST}^H)=(1 - l_t) Q_t(\mu^L, \pi_{\rm ST}^L,\pi_{\rm ST}^H) +  l_t \times \Big(\widehat{r}_t + \gamma  \max_{(\pi^{L\prime}_{\rm ST},\pi^{H\prime}_{\rm ST})\in (\tilde{\mathcal{P}}(\A))^2}Q_t(\mu^{L\prime}, \pi^{L\prime}_{\rm ST},\pi^{H\prime}_{\rm ST})\Big),
\end{equation}}

  \ENDFOR
  \ENDFOR
\ENDFOR
\ENDFOR
\end{algorithmic}
\end{algorithm}

\begin{Remark}
In general, distribution discretization is {\it sample inefficient} and suffers from the {\it curse of dimensionality}. For example, in Example \ref{wrong_Q}, there are  two states and two actions, with $N_s=N_a=20$  with precision $0.05$. The Q function is a table of dimension $8000$. This complexity grows exponentially with the number of states and actions. Moreover, although $E(t)$ converges relatively fast, there is unavoidable  errors due to truncation, as seen in Figure~\ref{fig:truncation_error}. The optimal value $Q_{t}\left(\frac{i}{N_s},\mbox{Proj}(\pi^{L, *}_{\rm ST}, \tilde\Pc(\A)), \mbox{Proj}(\pi^{H, *}_{\rm ST}, \tilde\Pc(\A))\right)$ can not be distinguished from its surrounding areas, where the areas with the lightest color all correspond to the largest value. This is because the accuracy is only up to $0.05$ in each iteration. Therefore, it is desirable to develop  sample-efficient and accurate Q-learning algorithms  for MFCs with learning with the correct Bellman equation \eqref{equ:BellmanQ}. See \cite{GGWX2020} for such a development with kernel regression method applied to improve the sample efficiency.
\end{Remark}

\begin{figure}[H]
  \centering
  \includegraphics[width=0.7\linewidth]{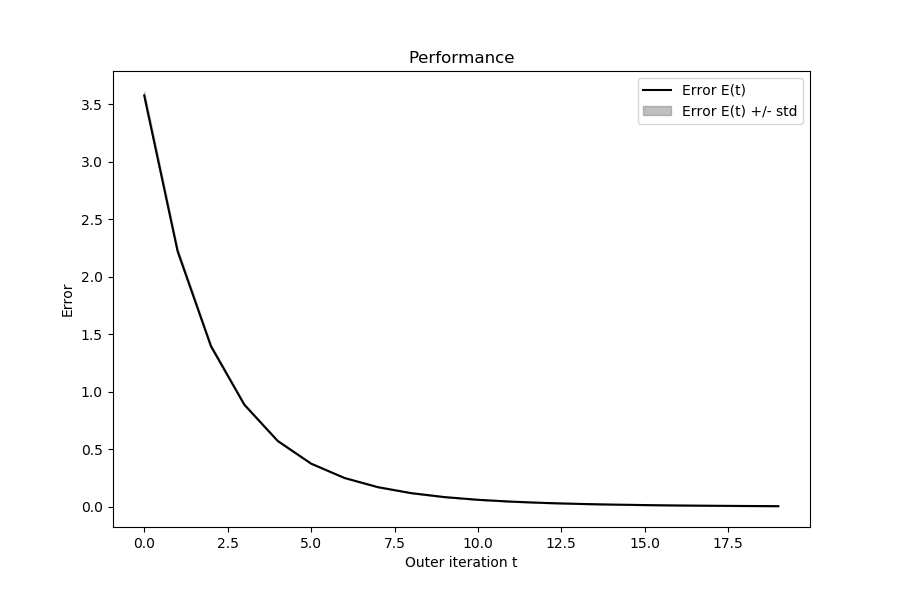}
  \caption{ \label{fig:errors}\centering{Numerical Performance on IQ Iterations.}}
\end{figure}

\begin{figure}[H]
\centering
\begin{subfigure}{.32\textwidth}
  \centering
  \includegraphics[width=0.99\linewidth]{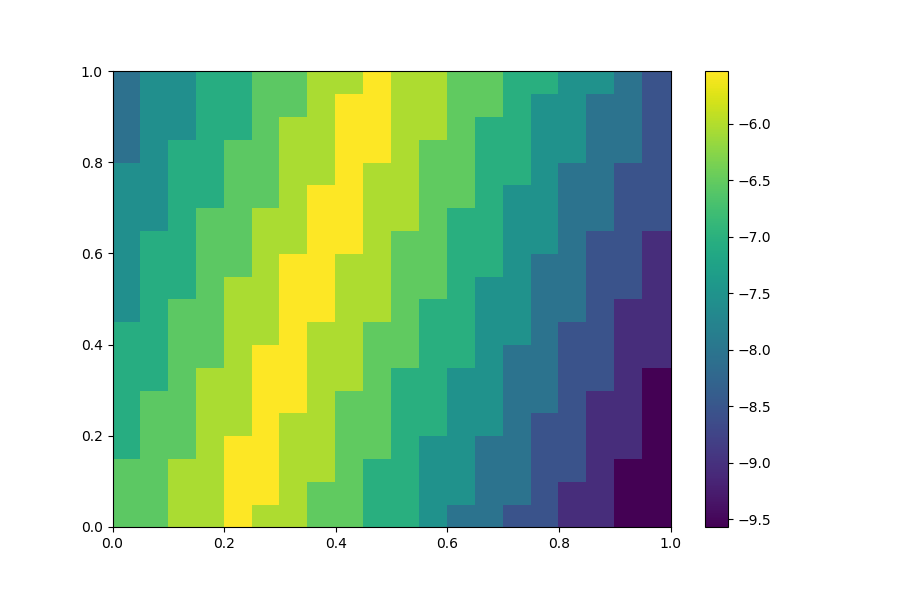}
  \centering
  \caption{$Q_T(0.3,\cdot,\cdot)$}
\end{subfigure}
\begin{subfigure}{.32\textwidth}
  \centering
  \includegraphics[width=0.99\linewidth]{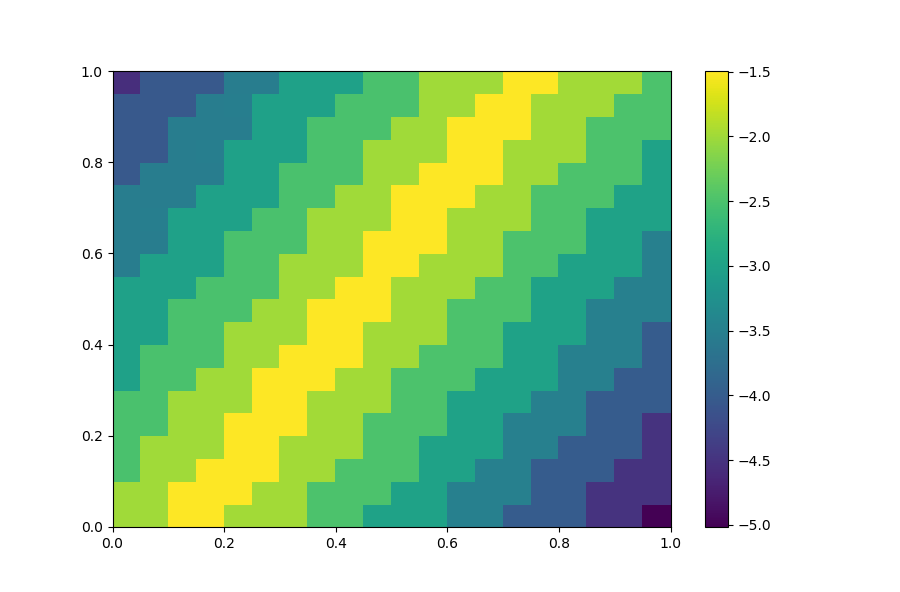}
  \caption{$Q_T(0.5,\cdot,\cdot)$.}
\end{subfigure}
\begin{subfigure}{.32\textwidth}
  \centering
  \includegraphics[width=0.99\linewidth]{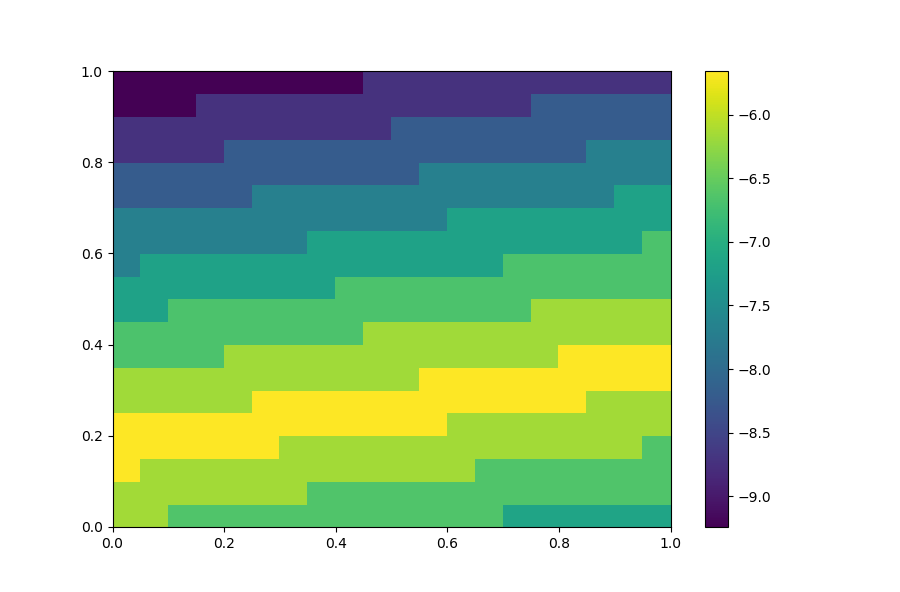}
  \caption{$Q_T(0.9,\cdot,\cdot)$}
\end{subfigure}
\caption{\label{fig:truncation_error}\centering{Snapshots of the IQ tables at final iteration $T$.}}
\end{figure}

{\begin{Remark}[Comparison with the time-inconsistent results in Section \ref{sec:wrong_Q}] The algorithm in Example \ref{wrong_Q} is designed based on the classical single-agent DPP with the usual state and action spaces $\mathcal{S}$ and $\mathcal{A}$. This approach fails both  theoretically and empirically in the mean-field regime. From a thoery point,  the classical  single-agent DPP is not ``rich enough'' to include all the necessary information. From an empirical perspective, the epsilon-greedy method and the time-dependent learning rate enables visiting each $(s,a)$ pair sufficiently many times, yet without the convergence guarantee.

In contrast, Algorithm \ref{MFCs_Q} finds the value of $Q(\mu,h)$ for any possible initial distribution $\mu$ including $\mu_0$ used in Example \ref{wrong_Q}. In addition, the convergence of the entire $Q$ table in Figure \ref{fig:errors} implies the convergence of $Q(\mu_0,\cdot)$ by the definition of $E(t)$.
\end{Remark}}     

\section{Additional example}
Consider a continuum of firms that supply a homogeneous product. For a representative firm, the sell price follows
\begin{eqnarray}\label{eq:price_dynamics}
s_t = s_{t-1} + \kappa(d_t-\mathbb{E}[a_t]) + w_t,
\end{eqnarray}
where $d_t$ is the (normalized) exogenous demand process per individual, $a_t$ is stochastic representing the supply volume from the representative firm,  $\{w_t\}_{t=1}^{\infty}$ are IID  noise following some distribution such as  the symmetric random walk, $\mathbb{E}[a_t]:=\int_{{\A}} a \nu_t(d\,a)$ is the aggregated supply volume from all firms where $\nu_t$ is the action distribution of all firms,
$\kappa>0$ is a scalar amplifying the impact from the supply-demand imbalance on the price process. Namely, the price process of the product will have a positive drift when demand is bigger than the supply whereas the price process will experience a negative drift if the average supply exceeds the demand. Correspondingly, the per-period reward accruing to the representative firm with supply volume $a_t$ is $$r_t = (s_t-c)a_t,$$
with $c>0$ the production cost.

\paragraph{Model set-up.} Assume $d_t\sim \mathcal{N}(2,0.25)$, $c=1$, $a_t \in \A: = \{0,1,2,\cdots,4\}$, and $\kappa=1$.  To enable Q-learning based algorithms, we truncate the values of the price dynamics within $s_t \in \mathcal{S}:=\{0,1,2,\cdots,19\}$.   Set the discount rate as $\gamma=0.6$ in the objective function. Finally, we consider $\{w_t\}_{t=1}^{\infty}$ follow IID random walks with probability $1/2$ being $1$ and  with probability $1/2$ being $-1$.

\paragraph{Design of the IQ table and the algorithm.} Recall that $s_t$ defined in \eqref{eq:price_dynamics} is the selling price received by the representative agent who produce $a_t$ amount of products during period $t$. Given the sets of actions and states specified above and from a population perspective, $\mu_t(i)$ denotes the proportion of firms who  received price $i-1$   and $\nu_t(i)$ denotes the proportion of firms taking action $i$ (i.e., supply with amount $i-1$) at time $t$.  
In addition, denote  $\hat{\mathcal{P}} (\mathcal{S}):=\big\{\mu\in \mathcal{P}(\mathcal{S})\,\, \mbox{ such that }\,\, \mu(s) \in \{j/N_s; j=0,1,\cdot,N_s \} \mbox{ for } s\in \mathcal{S}\big\}$ and  $\hat{\mathcal{H}}:=\big\{h\in \mathcal{H} \mbox{ such that } \,\, h(s;a) \in \{j/N_a; j=0,1,\cdot,N_a \}, \forall a\in \mathcal{A} \mbox{ and } s \in \mathcal{S}\big\}$ as the discretized probability measure spaces for the states and local policies, respectively. Therefore, it  is enough to consider the IQ table with the format of $Q(\mu,h)$ such that $\mu \in \hat{\mathcal{P}}(\mathcal{S})$ and $h\in  \hat{H}$. Recall the projection defined in \eqref{eq:projection},
$$\mbox{Proj}(\Phi(\mu^0,h),\hat \Pc(\Sc)):=\arg\Min_{\hat \mu \in\hat{\mathcal{P}}(\Sc)}|\Phi(\mu^0, h)-\hat{\mu}|.$$
We use this projection function to maintain the feasibility of the state distribution throughout training.
See Algorithm \ref{MFCs_Q_SG} for the detailed design of the learning algorithm, where we set $T=100$, $N_a = 20$, $N_s = 20$ and {a constant learning rate $l_t=l=0.1$}.
\begin{algorithm}[H]
  \caption{\textbf{MFCs Q-learning for the Supply Game}}
  \label{MFCs_Q_SG}
\begin{algorithmic}[1]
  \STATE \textbf{Input}: $N_a$.
 \FOR {$t= 1, \cdots,T$}
  \FOR {$h \in\hat{\mathcal{H}}$}
   \FOR {$\mu\in \hat{\mathcal{P}}(\mathcal{S})$}
 \STATE {$ \mu^{\prime}=\mbox{Proj}(\Phi^0(\mu,h),\hat{\mathcal{P}}(\mathcal{S}))$}
 \begin{equation} \label{eqn:IQ_update_reduction2}
Q_{t + 1}(\mu,h)=(1 - l_t) Q_t(\mu,h) +  l_t \times \Big(\hat{r}_t + \gamma  \max_{h^{\prime}\in \hat{\mathcal{H}}}Q_t(\mu^{\prime},h^{\prime}))\Big),
\end{equation}
\ENDFOR
\ENDFOR
\ENDFOR
\end{algorithmic}
\end{algorithm}

\paragraph{Results.} The IQ table converges with error less than $0.01$ within 60 outer iterations (see Figure \ref{fig:errors2}).

\begin{figure}[H]
  \centering
  \includegraphics[width=0.5\linewidth]{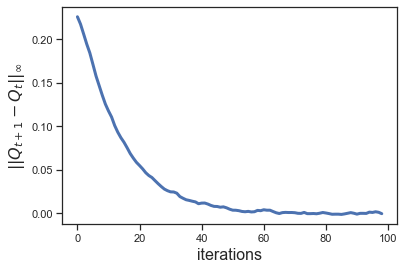}
  \caption{ \label{fig:errors2} \centering{ Convergence of the  IQ table.}}
\end{figure}

\begin{table}[h!]
\centering
\begin{tabular}{||c| c c c c c c c c c c||} 
 \hline
state (price $s$)  & 0 & 1 & 2 &3 &4 & 5 &6&7&8&9 \\ [0.5ex] 
 \hline
{{MFC}} solution & 0.4 & 0.65 & 0.9 & 0.8 & 1.15 & 0.8 & 1.25 & 0.9 & 0.95 &1.4 \\ 
\hline\hline
 state  (price $s$)  & 10 & 11 & 12 & 13 & 14 & 15 & 16& 17& 18&19 \\  \hline
{{MFC}} solution & 1.8 &
       2.05 & 2.15 & 1.9 & 2.3 & 3.  & 2.85 & 2.35 & 3.15 & 3.1  \\
\hline \hline
\end{tabular}
\caption{$\mathbb{E}[a^*(s)]$ in the MFC solution.}
\label{table:1}
\end{table}

Table \ref{table:1} shows the average supply from the learned MKV solution given different initial price.  When the price is small ($s_t=0$), it is optimal for the agents to provide a small amount of supplies to reduce the cost. 
When the price is in the middle range ($s_t=10$), it is optimal to suggest the allocation such that $\mathbb{E}[a_t]\approx 2$ with no impact on the price. When the price is high ($s_t=19$),  the price impact is tolerable by providing an excessive supply since it is highly profitable in this situation.

\paragraph{Comparison with Nash equilibrium} We also compare the performance of {MFC} solution under the Pareto optimality criterion with that of the mean-field game solution (MFG) under the Nash equilibrium criterion. The algorithm for learning the MFG solution is from \cite{GHXZ2019}. The output of the MFG strategy follows a Boltzmann type of policy  $\pi(s)(a)\sim \exp(\beta Q(s,a))$ with a temperature parameter $\beta>0$. Here we take $\beta=1$ and train the algorithm until the error falls below $10^{-2}$.

The trained Q table is provided in Figure \ref{nash_Q}, which indicates that in the equilibrium agents provide the largest supply (i.e., action 5) with a high probability. This is also consistent with the mean-field information  provided in Table \ref{table:2}. In the mean-field equilibrium, the expected supply is always bigger than $\mathbb{E}[d_t]=2$. This implies that, in a competitive market, agents are more aggressive in making and selling productions. 

In Figure \ref{reward_comparison}, we compare the cumulative rewards under the trained MFC policy with the trained MFG policy for 1000 rounds. We observe that the cumulative rewards from the MFC policy  is ten times bigger than those from the MFG policy. This implies that the aggressive behavior due to competition may leads to inefficiency from the market perspective, which indicates the necessity of understanding Pareto optimal solution for large-scale decision making problems.
\begin{figure}[H]
\centering
\begin{subfigure}{.45\textwidth}
  \centering
  \includegraphics[width=0.95\linewidth]{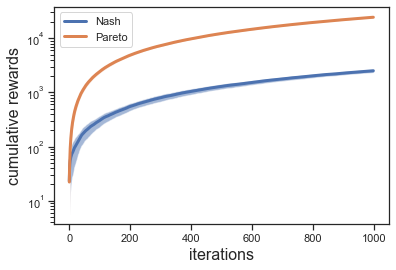}
  \centering
  \caption{\label{reward_comparison}Cumulative rewards for 1000 rounds.}
\end{subfigure}
\begin{subfigure}{.45\textwidth}
  \centering
  \includegraphics[width=0.95\linewidth]{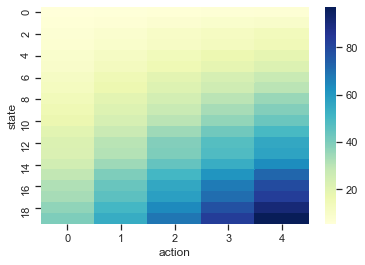}
  \caption{\label{nash_Q}Q table  of  the MFG solution.}
\end{subfigure}
\caption{\label{nash_vs_pareto}MFG solution vs. MFC solution.}
\end{figure}

\begin{table}[h!]
\centering
\begin{tabular}{||c c c c c c c c c c c||} 
 \hline
state (price $s$)  & 0 & 1 & 2 &3 &4 & 5 &6&7&8&9 \\ [0.5ex] 
 \hline
 MFG solution&2.08 & 2.19 & 2.37  & 2.37 & 2.51 &
       2.59& 2.75 & 2.81 & 2.92 & 3.01\\\hline\hline
 state  (price $s$)  & 10 & 11 & 12 & 13 & 14 & 15 & 16& 17& 18&19 \\  \hline
 MFG solution& 3.08 & 3.22 & 3.32 & 3.34 & 3.42 &
       3.51 & 3.56 & 3.60 & 3.65 & 3.68 \\\hline
 \hline
\end{tabular}
\caption{$\mathbb{E}[a^*(s)]$ in the MFG solution.}
\label{table:2}
\end{table}
\newpage


\bibliographystyle{informs2014}
\bibliography{refs}
\end{document}